\documentclass[11pt]{amsart}
\usepackage{amssymb}
\usepackage{amsmath}
\usepackage[active]{srcltx}
\usepackage{t1enc}
\usepackage[latin2]{inputenc}
\usepackage{verbatim}
\usepackage{amsmath,amsfonts,amssymb,amsthm}
\usepackage[mathcal]{eucal}
\usepackage{enumerate}
\usepackage[centertags]{amsmath}
\usepackage{graphics}

\newtheorem{theorem}{Theorem}

\newtheorem{lemma}{Lemma}
\newtheorem{remark}{Remark}

\newtheorem{corollary}{Corollary}

\begin{document}
\author{G. Tephnadze}
\title[partial sums]{On the convergence of partial sums with respect to Vilenkin system on the martingale Hardy spaces}
\address{G. Tephnadze, The University of Georgia, IV, 77a Merab Kostava St, Tbilisi, 0128,  Georgia, \& Department of Engineering Sciences and Mathematics, Lule\aa University of Technology, SE-971 87 Lule\aa, Sweden.}
\email{giorgitephnadze@gmail.com}
\thanks{The research was supported by Shota Rustaveli National Science Foundation grants no. DO/24/5-100/14 and YS15-2.1.1-47, by a
Swedish Institute scholarship no. 10374-2015 and by target scientific research programs
grant for the students of faculty of Exact and Natural Sciences.}
\date{}
\maketitle

\begin{abstract}
In this paper we derive characterizations of boundedness  of the subsequences of partial sums with respect to Vilenkin system on the martingale Hardy spaces when $ 0<p<1 $. Moreover, we find necessary and sufficient conditions for the modulus of continuity of $f\in H_{p}$ martingales, which provide convergence of subsequences of partial sums on the martingale Hardy spaces. 
It is also proved that these results are the best possible in a special sense. As applications, both some well-known and new results are pointed out.
\end{abstract}

\date{}

\textbf{2010 Mathematics Subject Classification.} 42C10.

\textbf{Key words and phrases:} Vilenkin system, partial sums, martingale
Hardy space, modulus of continuity, convergence.

\section{ Introduction}

The definitions and notations used in this introduction can be found in our next
Section. It is well-known that (for details see e.g. \cite{sws})
\begin{equation*}
\left\Vert S_{n}f\right\Vert _{p}\leq c_{p}\left\Vert f\right\Vert _{p},%
\text{ \ when \ }p>1,
\end{equation*}%
where $ S_{n}f $ is $ n $-th partial sum with respect to bounded Vilenkin system.

Moreover, it can be proved also a more stronger result (for details see e.g.  \cite{gol}):
\begin{equation*}
\left\Vert S^{\ast }f\right\Vert _{p}\leq c_{p}\left\Vert f\right\Vert _{p},%
\text{ \ when \ }f\in L_{p},\text{ \ }p>1,
\end{equation*}
where
\begin{equation*}
{S}^{\ast }f:=\sup_{n\in \mathbb{N}}\left\vert
S_{n}f\right\vert.
\end{equation*}%

Lukomskii \cite{luko} proved two-side estimate
for Legesgue constants $L_{n} $ with respect to Vilenkin sistem. By using
this result we easily obtain that $S_{n_{k}}f$\ convergence to $f$ \ in $%
L_{1} $-norm, for every integrable function $f$, if and only if
\begin{equation*}
\sup_{k \in \mathbb{N}}L_{n_{k}}\leq c<\infty .
\end{equation*}

Point-wise and uniform convergence and some approximation properties of
partial sums in $L_{1} $ norm were studied by a lot of
autors. We mentioned, for instance, the paper of Goginava \cite{gog1}, Goginava and Sahakian \cite{gs} and Avdispahi\'c and Memi\'c \cite{am}. Fine \cite{fi} obtained sufficient conditions for the uniform convergence which are in
complete analogy with the Dini-Lipschits conditions. Guličev \cite{9}
estimated the rate of uniform convergence of a Walsh-Fourier series by using Lebesgue constants and modulus of continuity. Uniform convergence of
subsequence of partial sums with respect to Walsh system was investigate also in \cite{gt}. This problem for Vilenkin group $G_{m}$ was considered by Blahota \cite{2},  Fridli \cite{4} and Gát \cite{5}.

It is known (for details see e.g. \cite{tep5}) that Vilenkin system
does not form  basis in the space $L_{1}\left( G_{m}\right) .$ Moreover, there is
a function in the martingale Hardy space $H_{1}\left( G_{m}\right) ,$ such
that the partial sums of $f$ are not bounded in $L_{1}\left( G_{m}\right) $%
-norm, but subsequence $S_{M_{n}}$ of partial sums  are bounded from the
martingale Hardy space $H_{p}\left( G_{m}\right) $ to the Lebesgue space $%
L_{p}\left( G_{m}\right) ,$ for all $p>0.$

In \cite{tepthesis} it was proved that if $0<p\leq 1$
and $\left\{ \alpha _{k}:k\in \mathbb{N}\right\} $ be a increasing secuence
of nonnegative integers such that
\begin{equation}  \label{ak0}
\sup_{k\in \mathbb{N}}\rho \left( \alpha _{k}\right) <\infty,
\end{equation}
where
\begin{equation*}
\rho\left( n\right)=\left\vert n\right\vert -\left\langle n\right\rangle
\end{equation*}
and

\begin{equation*}
\left\langle n\right\rangle :=\min \{j\in \mathbb{N}:n_{j}\neq 0\}\text{,\ }%
\left\vert n\right\vert :=\max \{j\in \mathbb{N}:n_{j}\neq 0\},
\end{equation*}
for 
$n=\overset{\infty
}{\underset{k=0}{\sum }}n_{j}M_{j}, \ \ n_{j}\in Z_{m_{j}} \ \ (j\in
\mathbb{N}),$
then the restricted maximal operator
\begin{equation*}
\widetilde{S}^{\ast ,\vartriangle }f:=\sup_{k\in \mathbb{N}}\left\vert
S_{\alpha _{k}}f\right\vert
\end{equation*}%
is bounded from the Hardy space $H_{p}$ to the Lebesgue space $L_{p}.$

Moreover, if $0<p<1$ and $\left\{ \alpha _{k}:k\in \mathbb{N}\right\} $ be a
increasing secuence of nonnegative integers satisfying the condition
\begin{equation}  \label{ak}
\sup_{k\in \mathbb{N}}\rho \left( \alpha _{k}\right) =\infty,
\end{equation}
then there exists a martingale $f\in H_{p}$ such that
\begin{equation*}
\sup_{k\in \mathbb{N}}\left\Vert S_{\alpha _{k}}f\right\Vert _{L_{p,\infty
}}=\infty.
\end{equation*}

It immediately follows that for any $p>0$ and $f\in H_{p}$, the following restricted maximal operator
\begin{equation*}
\widetilde{S}_{\#}^{\ast }f:=\sup_{n\in \mathbb{N}}\left\vert
S_{M_{n}}f\right\vert,
\end{equation*}
where
$M_{0}:=1,$ $M_{k+1}:=\overset{k}{\underset{i=0}{\Pi }}m_{i}$ and $m:=(m_{0,}$ $m_{1},...)$ be a sequence of positive integers not less than 2, which generates a Vilenkin system,
is bounded from the Hardy space $H_{p}$ to the space $L_{p}$:
\begin{equation}  \label{Sk0}
\left\Vert\widetilde{S}_{\#}^{\ast }f\right\Vert_{p}\leq \left\Vert f
\right\Vert_{H_p}, \text{ \ } f\in H_{p}.
\end{equation}

For the Vilenkin system Simon \cite{si1} proved that there is an absolute
constant $c_{p},$ depending only on $p,$ such that
\begin{equation}
\overset{\infty }{\underset{k=1}{\sum }}\frac{\left\Vert S_{k}f\right\Vert
_{p}^{p}}{k^{2-p}}\leq c_{p}\left\Vert f\right\Vert _{H_{p}}^{p},
\label{1cc}
\end{equation}%
for all $f\in H_{p}\left( G_{m}\right) $, where $0<p<1.$ In \cite{tep3} we proved that sequence $\left\{ 1/k^{2-p}:k\in \mathbb{N}\right\} $ can
not be improved.

Analogical theorem for $p=1$ with respect to the unbounded Vilenkin systems was
proved in Gát \cite{gat1}.

In \cite{tep5} we proved that if $0<p<1,$ $f\in H_{p}\left( G_{m}\right) $
and%
\begin{equation*}
\omega \left( \frac{1}{M_{n}},f\right) _{H_{p}\left( G_{m}\right) }=o\left(
\frac{1}{M_{n}^{1/p-1}}\right) ,\text{ as }n\rightarrow \infty ,
\end{equation*}%
then%
\begin{equation}  \label{con1p}
\left\Vert S_{n_{k}}f-f\right\Vert _{H_{p}}\rightarrow 0,\,\,\,\text{as}%
\,\,\,k\rightarrow \infty.
\end{equation}%

Moreover, for every $p\in \left( 0,1\right) $ there exists martingale $f\in
H_{p}(G_{m})$, for which
\begin{equation*}
\omega \left( \frac{1}{M_{n}},f\right) _{H_{p}(G_{m})}=O\left( \frac{1}{%
M_{n}^{1/p-1}}\right) ,\text{ as }n\rightarrow \infty
\end{equation*}%
and
\begin{equation*}
\left\Vert S_{k}f-f\right\Vert _{L_{p,\infty
}(G_{m})}\nrightarrow 0,\,\,\,\text{as\thinspace \thinspace \thinspace }%
k\rightarrow \infty .
\end{equation*}

In \cite{tep8} we investigated some $(H_{p},H_{p}),$ $(H_{p},L_{p})$ and
$(H_{p},L_{p,\infty })$ type inequalities of subsequences of partial sums with respect to
Walsh-Fourier series for $0<p\leq 1.$

In this paper we derive characterizations of boundedness (or even the ratio of divergence of the norm) of the subsequences of partial sums with respect to Vilenkin system on the martingale Hardy spaces when $ 0<p<1 $. Moreover, we find necessary and sufficient conditions for the modulus of continuity of $f\in H_{p}$, which provide convergence of subsequences of partial sums on the martingale Hardy spaces. 
It is also proved that these results are the best possible in a special sense. As applications, both some well-known and new results are pointed out.

This paper is organized as follows: In order not to disturb our discussions later on some definitions and notations are presented in Section 2. For the proofs of the main results we need some
auxiliary Lemmas, some of them are new and of independent interest. These results are also
presented in Section 2. The formulations and detailed proofs of our main results and some of its
consequences can be found in Sections 3 and 4.  

\section{Preliminaries}

Let $\mathbb{N}_{+}$ denote the set of the positive integers, $\mathbb{N}:=%
\mathbb{N}_{+}\cup \{0\}.$

Let $m:=(m_{0,}$ $m_{1},...)$ denote a sequence of the positive integers not
less than 2. Denote by $Z_{m_{k}}:=\{0,1,...,m_{k}-1\}$ the additive group
of integers modulo $m_{k}.$ Define the group $G_{m}$ as the complete direct
product of the group $Z_{m_{j}}$ with the product of the discrete topologies
of $Z_{m_{j}}$ s. The direct product $\mu $ of the measures 
$$\mu _{k}\left(\{j\}\right) :=1/m_{k} \text{ }(j\in Z_{m_{k}})$$  
is the Haar measure on $G_{m_{%
\text{ }}}$ with $\mu \left( G_{m}\right) =1.$ 

If the sequence $m:=(m_{0,}$ 
$m_{1},...)$ is bounded than $G_{m}$ is called a bounded Vilenkin group, else
it is called an unbounded one.

The elements of $G_{m}$ are represented by sequences
\begin{equation*}
x:=(x_{0},x_{1},...,x_{j},...)\ \left( \text{ }x_{k}\in Z_{m_{k}}\right) .
\end{equation*}

It is easy to give a base for the neighborhood of $G_{m}$
\begin{eqnarray*}
I_{0}\left( x\right) &:&=G_{m}, \\
I_{n}(x) &:&=\{y\in G_{m}\mid y_{0}=x_{0},...,y_{n-1}=x_{n-1}\}\text{ }(x\in
G_{m},\text{ }n\in \mathbb{N}).
\end{eqnarray*}

Denote $I_{n}:=I_{n}\left( 0\right) $ for $n\in \mathbb{N}$ and $\overline{%
I_{n}}:=G_{m} \setminus I_{n}$. It is evident that

\begin{equation}  \label{1.1}
\overline{I_{N}}=\overset{N-1}{\underset{s=0}{\bigcup }}I_{s}\setminus
I_{s+1}.
\end{equation}

If we define the so-called generalized number system based on $m$ in the
following way $M_{0}:=1,$ $M_{k+1}:=m_{k}M_{k\text{ }}\ (k\in \mathbb{N})$
then every $n\in \mathbb{N}$ can be uniquely expressed as $n=\overset{\infty
}{\underset{k=0}{\sum }}n_{j}M_{j},$ where $n_{j}\in Z_{m_{j}}$ $~(j\in
\mathbb{N})$ and only a finite number of $n_{j}$`s differ from zero.

For all $n\in \mathbb{N}$ let us define
\begin{equation*}
\left\langle n\right\rangle :=\min \{j\in \mathbb{N}:n_{j}\neq 0\}\text{,\ }%
\left\vert n\right\vert :=\max \{j\in \mathbb{N}:n_{j}\neq 0\}\text{, \ }%
\rho\left( n\right) =\left\vert n\right\vert -\left\langle n\right\rangle .
\end{equation*}

For the natural number $n=\sum_{j=1}^{\infty }n_{j}M_{j},$ we define
functions $v$ and $v^{\ast }$ by

\begin{equation*}
v\left( n\right) =\sum_{j=1}^{\infty }\left\vert \delta _{j+1}-\delta
_{j}\right\vert +\delta _{0},\text{ \ }v^{\ast }\left( n\right)
=\sum_{j=1}^{\infty }\delta _{j}^{\ast },
\end{equation*}%
where
\begin{equation*}
\delta_{j}=signn_{j}=sign\left( \ominus n_{j}\right), \text{ \ }
\delta_{j}^{\ast }=\left\vert \ominus n_{j}-1\right\vert \delta_{j}
\end{equation*}
and $\ominus $\ is the inverse operation for $a_{k}\oplus
b_{k}:=(a_{k}+b_{k})$mod$m_{k}.$

The norms (or quasi-norms) of the spaces $L_{p}(G_{m})$ and $L_{p,\infty
}\left( G_{m}\right) $ $\left( 0<p<\infty \right) $ are respectively defined
by \qquad \qquad \thinspace\
\begin{eqnarray*}
&&\left\Vert f\right\Vert _{p}^p:=\int_{G_{m}}\left\vert f\right\vert
^{p}d\mu, \\
&&\left\Vert f\right\Vert _{L_{p,\infty}}:=\underset{\lambda >0}{\sup }\lambda
\mu \left( f>\lambda \right) ^{1/p}.
\end{eqnarray*}

Next, we introduce on $G_{m}$ an orthonormal system which is called the
Vilenkin system.

First define the complex valued function $r_{k}\left( x\right)
:G_{m}\rightarrow
\mathbb{C}
,$ the generalized Rademacher functions as
\begin{equation*}
r_{k}\left( x\right) :=\exp \left( 2\pi ix_{k}/m_{k}\right) \text{ \qquad }%
\left( \iota ^{2}=-1,\text{ }x\in G_{m},\text{ }k\in \mathbb{N}\right) .
\end{equation*}

Now define the Vilenkin system $\psi :=(\psi _{n}:n\in \mathbb{N})$ on $%
G_{m} $ as:
\begin{equation*}
\psi _{n}(x):=\overset{\infty }{\underset{k=0}{\Pi }}r_{k}^{n_{k}}\left(
x\right) \text{ \qquad }\left( n\in \mathbb{N}\right) .
\end{equation*}

Specifically, we call this system the Walsh-Paley one if $m\equiv 2,$ that is $m_k=2$ for all $ k\in \mathbb{N} $.

The Vilenkin system is orthonormal and complete in $L_{2}\left( G_{m}\right)
\,$ (for details see e.g. \cite{AVD,Vi}).

If $f\in L_{1}\left( G_{m}\right) $ we can establish the Fourier
coefficients, the partial sums of the Fourier series, the Dirichlet kernels
with respect to the Vilenkin system $\psi $ in the usual manner:
\begin{eqnarray*}
\widehat{f}\left( k\right) &:&=\int_{G_{m}}f\overline{\psi }_{k}d\mu ,\text{%
\thinspace }\text{ \ \ } \left(k\in \mathbb{N}\right) \\
S_{n}f &:&=\sum_{k=0}^{n-1}\widehat{f}\left( k\right) \psi _{k},\text{ \ \ }\left(k\in \mathbb{N}\right) \\
D_{n}&:&=\sum_{k=0}^{n-1}\psi _{k\text{ }},\text{ } \left(n\in \mathbb{N}%
_{+}\right).
\end{eqnarray*}

Recall that (for details see e.g. \cite{AVD})
\begin{equation}
\quad \hspace*{0in}D_{M_{n}}\left( x\right) =\left\{
\begin{array}{l}
M_{n},\text{ \ \ }x\in I_{n} \\
0,\text{ \ \ }x\notin I_{n}%
\end{array}%
\right.  \label{3dn}
\end{equation}%
\vspace{0pt}and

\begin{equation}
D_{n}=\psi _{n}\left( \underset{j=0}{\overset{\infty }{\sum }}D_{M_{j}}%
\overset{m_{j}-1}{\underset{u=m_{j}-n_{j}}{\sum }}r_{j}^{u}\right) .
\label{2dn}
\end{equation}

Moreover, (for details see Tephnadze \cite{tep2} and \cite{tep6}) if $n\in \mathbb{N}$ and $x\in I_{s}\setminus
I_{s+1},$ $0\leq s\leq N-1,$ then the following apper and bellow etimates are true:
\begin{equation} \label{estdn0}
\left\vert D_{n}(x)\right\vert =\left\vert D_{n-M_{\left\vert n\right\vert
}}(x)\right\vert \geq M_{\left\langle n\right\rangle }, \text{ \ \ \ \ \ }\left\vert n\right\vert \neq
\left\langle n\right\rangle 
\end{equation}
and 
\begin{equation} \label{estdn}
\int_{I_{N}}\left\vert D_{n}\left( x-t\right) \right\vert d\mu \left(
t\right) \leq \,\frac{cM_{s}}{M_{N}}.
\end{equation}

The $n$-th Lebesgue constant is defined in the following way:
\begin{equation*}
L_{n}:=\left\Vert D_{n}\right\Vert _{1}.
\end{equation*}

It is known that (see Lukomskii \cite{luko}) that for every $n=\sum_{i=1}^{\infty }n_{i}M_{i}$, the following two-side estimate is true:
\begin{equation} \label{var}
\frac{1}{4\lambda }v\left( n\right) +\frac{1}{\lambda }v^{\ast }\left(
n\right) +\frac{1}{2\lambda }\leq L_{n}\leq \frac{3}{2}v\left( n\right)
+4v^{\ast }\left( n\right) -1,
\end{equation}%
where $\lambda :=\sup_{n\in \mathbb{N}}m_{n}.$

The $\sigma $-algebra generated by the intervals $\left\{ I_{n}\left(
x\right) :x\in G_{m}\right\} $ will be denoted by $\digamma _{n}$ $\left(
n\in \mathbb{N}\right) .$ Denote by $f=\left( f_{n},n\in \mathbb{N}\right) $
a martingale with respect to $\digamma _{n}$ $\left( n\in \mathbb{N}\right)
. $ (for details see e.g. Weisz \cite{We1}).

The maximal function of a martingale $f$ is defined by
\begin{equation*}
f^{\ast }=\sup_{n\in \mathbb{N}}\left\vert f^{\left( n\right) }\right\vert .
\end{equation*}

In case $f\in L_{1}\left( G_{m}\right) ,$ the maximal functions are also be
given by
\begin{equation*}
f^{\ast }\left( x\right) =\sup_{n\in \mathbb{N}}\frac{1}{\left\vert
I_{n}\left( x\right) \right\vert }\left\vert \int_{I_{n}\left( x\right)
}f\left( u\right) d\mu \left( u\right) \right\vert
\end{equation*}

For $0<p<\infty $ the Hardy martingale spaces $H_{p}\left( G_{m}\right) $
consist of all martingales, for which
\begin{equation*}
\left\Vert f\right\Vert _{H_{p}}:=\left\Vert f^{\ast }\right\Vert
_{p}<\infty .
\end{equation*}

Let $X=X(G_{m})$ denote either the space $L_{1}(G_{m}),$ or the space of
continuous functions $C(G_{m})$. The corresponding norm is denoted by $\Vert
.\Vert _{X}$. The modulus of continuity, when $X=C\left( G_{m}\right) $ and
the integrated modulus of continuity, where $X=L_{1}\left( G_{m}\right) $
are defined by

\begin{equation*}
\omega \left( \frac{1}{M_{n}},f\right) _{X}=\sup\limits_{h\in I_{n}}\left\Vert
f\left( \cdot +h\right) -f\left( \cdot \right) \right\Vert _{X}.
\end{equation*}

The concept of modulus of continuity in $H_{p}\left( G_{m}\right) \left(
0<p\leq 1\right) $ can be defined in following way
\begin{equation*}
\omega \left(\frac{1}{M_{n}},f\right) _{H_{p}\left( G_{m}\right) }:=\left\Vert
f-S_{M_{n}}f\right\Vert _{H_{p}\left( G_{m}\right) }.
\end{equation*}

If $f\in L_{1}\left( G_{m}\right) ,$ then it is easy to show that the
sequence $\left( S_{M_{n}}f :n\in \mathbb{N}\right) $ is a martingale.

If $f=\left( f_{n},n\in \mathbb{N}\right) $ is a martingale then the
Vilenkin-Fourier coefficients must be defined in a slightly different
manner: $\qquad \qquad $
\begin{equation*}
\widehat{f}\left( i\right) :=\lim_{k\rightarrow \infty
}\int_{G_{m}}f_{k}\left( x\right) \overline{\psi }_{i}\left( x\right) d\mu
\left( x\right) .
\end{equation*}%
\qquad \qquad \qquad

The Vilenkin-Fourier coefficients of $f\in L_{1}\left( G_{m}\right) $ are
the same as the martingale $\left( S_{M_{n}}f :n\in \mathbb{N}\right) $
obtained from $f$ .

A bounded measurable function $a$ is p-atom, if there exist an
interval $I$, such that%
\begin{equation*}
\int_{I}ad\mu =0,\text{ \ \ \ \ }\left\Vert a\right\Vert _{\infty }\leq \mu
\left( I\right) ^{-1/p},\text{ \ \ \ \ \ supp}\left( a\right) \subset I.
\end{equation*}

The martingale Hardy spaces $H_{p}$ $\left( G_{m}\right) $ for $0<p\leq 1$ have
atomic characterizations (for details see e.g. Weisz \cite{We1} and \cite{We3}):

\begin{lemma}
\label{lemma2.1} A martingale $f=\left( f_{n},n\in \mathbb{N}\right) $ is in
$H_{p}\left( 0<p\leq 1\right) $ if and only if there exist a sequence $%
\left( a_{k},k\in \mathbb{N}\right) $ of p-atoms and a sequence $\left( \mu
_{k},k\in \mathbb{N}\right) $ of real numbers such that, for every $n\in
\mathbb{N},$%
\begin{equation}
\qquad \sum_{k=0}^{\infty }\mu _{k}S_{M_{n}}a_{k}=f_{n},\text{ \ \ a.e.,}
\label{condmart}
\end{equation}%
where
\begin{equation*}
\qquad \sum_{k=0}^{\infty }\left\vert \mu _{k}\right\vert ^{p}<\infty .
\end{equation*}%
Moreover,
\begin{equation*}
\left\Vert f\right\Vert _{H_{p}}\backsim \inf \left( \sum_{k=0}^{\infty
}\left\vert \mu _{k}\right\vert ^{p}\right) ^{1/p},
\end{equation*}
where the infimum is taken over all decomposition of $f$ of the form (\ref%
{condmart}).
\end{lemma}

By using atomic decomposition of $ f\in H_p $ martingales, we can bring counterexample, which play a central role to prove
sharpness of our main results and it will be used several times in the paper
(for details see e.g Tephnadze \cite{tepthesis}, Section 1.7., Example 1.48):

\begin{lemma}
\label{example2.6}Let $0<p\leq 1,$ $\lambda =\sup_{n\in \mathbb{N}}m_{n},$ $%
\left\{ \lambda _{k}:k\in \mathbb{N}\right\} $ be a sequence of real
numbers, such that%
\begin{equation}
\sum_{k=0}^{\infty }\left\vert \lambda _{k}\right\vert ^{p}\leq c_{p}<\infty
.  \label{2aa}
\end{equation}%
and $\left\{ a_{k}:k\in \mathbb{N}\right\} $ be a sequence of $p$-atoms,
defined by%
\begin{equation*}
a_{k}:=\frac{M_{\left\vert \alpha _{k}\right\vert }^{1/p-1}}{\lambda }\left(
D_{M_{\left\vert \alpha _{k}\right\vert +1}}-D_{M_{_{\left\vert \alpha
_{k}\right\vert }}}\right) ,
\end{equation*}%
where $\left\vert \alpha _{k}\right\vert :=\max $ $\{j\in \mathbb{N}:$ $%
\left( \alpha _{k}\right) _{j}\neq 0\}$ and $\left( \alpha _{k}\right) _{j}$
denotes the $j$-th binary coefficient of $\alpha _{k}\in \mathbb{N}.$ Then $%
\,f=\left( f_{n}:\text{ }n\in \mathbb{N}\right) ,$ where \qquad
\begin{equation*}
f_{n}:=\sum_{\left\{ k:\text{ }\left\vert \alpha _{k}\right\vert <n\right\}
}\lambda _{k}a_{k}
\end{equation*}
is a martingale, $f\in H_{p}\ $ for all $0<p\leq 1$ and
\begin{eqnarray}   \label{10AA}
&&\widehat{f}(j)=\left\{
\begin{array}{ll}
\frac{\lambda _{k}M_{\left\vert \alpha _{k}\right\vert }^{1/p-1}}{\lambda },\text{\ \ \ } j\in \left\{ M_{\left\vert \alpha _{k}\right\vert },...,\text{ ~}%
M_{\left\vert \alpha _{k}\right\vert +1}-1\right\} ,\text{ }k\in \mathbb{N}, \\
0,  \text{\ \ \ \ \ \ \ \ \ \ \ \ }j\notin \bigcup\limits_{k=1}^{\infty }\left\{
M_{\left\vert \alpha _{k}\right\vert },...,\text{ ~}M_{\left\vert \alpha
_{k}\right\vert +1}-1\right\} .%
\end{array}%
\right.
\end{eqnarray}

Let $M_{\left\vert \alpha _{l}\right\vert }\leq j<M_{\left\vert \alpha
_{l}\right\vert +1},$ $l\in \mathbb{N}.$ Then%
\begin{eqnarray}  \label{11AA}
S_{j}f&=& S_{M_{\left\vert \alpha _{l}\right\vert }}+\frac{\lambda
_{l}M_{\left\vert \alpha _{l}\right\vert }^{1/p-1}\psi _{M_{\left\vert
\alpha _{l}\right\vert }}D_{j-M_{_{\left\vert \alpha _{l}\right\vert }}}}{%
\lambda } \\
&=&\sum_{\eta =0}^{l-1}\frac{\lambda _{\eta }M_{\left\vert \alpha _{\eta
}\right\vert }^{1/p-1}}{\lambda }\left( D_{M_{_{\left\vert \alpha _{\eta
}\right\vert +1}}}-D_{M_{_{\left\vert \alpha _{\eta }\right\vert }}}\right)
\notag \\
&+&\frac{\lambda _{l}M_{\left\vert \alpha _{l}\right\vert }^{1/p-1}\psi
_{M_{\left\vert \alpha _{l}\right\vert }}D_{j-M_{\left\vert \alpha
_{l}\right\vert }}}{\lambda }.  \notag
\end{eqnarray}

Moreover,
\begin{equation}
\omega \left( \frac{1}{M_{n}},f\right)_{H_{p}\left( G_{m}\right) } =O\left( \sum_{\left\{ k:%
\text{ }\left\vert \alpha _{k}\right\vert \geq n\right\} }^{\infty
}\left\vert \lambda _{k}\right\vert ^{p}\right) ^{1/p},\text{\ \ as \ \ }%
n\rightarrow \infty,  \label{2aa0}
\end{equation}
\end{lemma}

There exists an intimate conection between the $ H_p $ and $ L_p $ norms of partial sums (for details see e.g. Tephnadze \cite{tepthesis}, Section 1.7., Example1.45):

\begin{lemma}
\label{example002}Let $M_{k}\leq n<M_{k+1}$ and $S_{n}f$ be the n-th partial
sum with respect to Vilenkin system, where $f\in H_{p}$ \ for some $0<p\leq
1.$ Then for every $n\in \mathbb{N}$ we have the following estimate:
\begin{eqnarray*}
\left\Vert
S_{n}f\right\Vert _{p}\leq\left\Vert S_{n}f\right\Vert _{H_{p}} &\leq &\left\Vert \sup_{0\leq l\leq
k}\left\vert S_{M_{l}}f\right\vert \right\Vert _{p}+\left\Vert
S_{n}f\right\Vert _{p} \\
&\leq &\left\Vert \widetilde{S}_{\#}^{\ast }f\right\Vert _{p}+\left\Vert
S_{n}f\right\Vert _{p}.
\end{eqnarray*}
\end{lemma}

\qquad

\section{Convergence of subsequences of partial sums on the martingale Hardy spaces}

Our first main result reads:

\begin{theorem}
\label{theorem4.15}a) Let $0<p<1$ and $f\in H_{p}$. Then there exists an
\textit{absolute constant }$c_{p},$ depending only on $p,$ such that%
\begin{equation*}
\text{ }\left\Vert S_{n}f\right\Vert _{H_{p}}\leq \frac{c_{p}M_{\left\vert
n\right\vert }^{1/p-1}}{M_{\left\langle n\right\rangle }^{1/p-1}}\left\Vert
f\right\Vert _{H_{p}}.
\end{equation*}

b) Let $0<p<1\ \ $and $\left\{ n_{k}:\text{ }k\in \mathbb{N}\right\} $ be an increasing sequence of nonnegative integers such that condition (\ref%
{ak}) is satisfied and $\left\{ \Phi _{n}:n\in \mathbb{N}\right\} $ be any
non-decreasing sequence, satisfying the condition
\begin{equation} \label{12e}
\overline{\underset{k\rightarrow \infty }{\lim }}\frac{M_{\left\vert
n_{k}\right\vert }^{1/p-1}}{M_{\left\langle n_{k}\right\rangle }^{1/p-1}\Phi
_{n_{k}}}=\infty .  
\end{equation}%
Then there exists a martingale $f\in H_{p},$ such that
\begin{equation*}
\underset{k\in \mathbb{N}}{\sup }\left\Vert \frac{S_{n_{k}}f}{\Phi _{n_{k}}}%
\right\Vert _{L_{p,\infty }}=\infty .
\end{equation*}
\end{theorem}

\begin{proof}
Suppose that
\begin{equation}
\left\Vert \frac{M_{\left\langle n\right\rangle }^{1/p-1}S_{n}f}{%
M_{\left\vert n\right\vert }^{1/p-1}}\right\Vert _{p}\leq c_{p}\left\Vert
f\right\Vert _{H_{p}}.  \label{11.1}
\end{equation}

According to Lemma \ref{example002} and estimates (\ref{Sk0}) and (\ref{11.1}%
) we get that%
\begin{eqnarray}  \label{11.2}
\left\Vert \frac{M_{\left\langle n\right\rangle }^{1/p-1}S_{n}f}{%
M_{\left\vert n\right\vert }^{1/p-1}}\right\Vert _{H_{p}} \leq \left\Vert
\widetilde{S}_{\#}^{\ast }f\right\Vert _{p}+\left\Vert \frac{M_{\left\langle
n\right\rangle }^{1/p-1}S_{n}f}{M_{\left\vert n\right\vert }^{1/p-1}}%
\right\Vert _{p}\leq c_{p}\left\Vert f\right\Vert _{H_{p}}.
\end{eqnarray}

By using Lemma \ref{lemma2.1} and (\ref{11.2}) the proof of part a) of
Theorem \ref{theorem4.15} will be complete, if we show that%
\begin{equation}
\int\limits_{G}\left\vert \frac{M_{\left\langle n\right\rangle
}^{1/p-1}S_{n}a}{M_{\left\vert n\right\vert }^{1/p-1}}\right\vert d\mu \leq
c_{p}<\infty ,  \label{25a}
\end{equation}%
for every p-atom $a,$ with support$\ I$ and $\mu \left( I\right) =M_{N}^{-1}$%
.

We may assume that this arbitrary p-atom $a$ has support $I=I_{N}.$ It is
easy to see that $S_{n}a =0,$ when $M_{N}$ $\geq n$. Therefore, we can
suppose that $M_{N}<n$. According to $\left\Vert a\right\Vert _{\infty }\leq
M_{N}^{1/p}$  we can write that
\begin{eqnarray} \label{11a1}
\left\vert \frac{M_{\left\langle n\right\rangle }^{1/p-1}S_{n}a\left(
x\right) }{M_{\left\vert n\right\vert }^{1/p-1}}\right\vert &\leq &\frac{%
M_{\left\langle n\right\rangle }^{1/p-1}\left\Vert a\right\Vert _{\infty }}{%
M_{\left\vert n\right\vert }^{1/p-1}}\int_{I_{N}}\left\vert
D_{n}\left(x-t\right) \right\vert d\mu \left( t\right)   \\
&\leq &\frac{M_{\left\langle n\right\rangle }^{1/p-1}M_{N}^{1/p}}{%
M_{\left\vert n\right\vert }^{1/p-1}}\int_{I_{N}}\left\vert D_{n}\left(
x-t\right) \right\vert d\mu \left( t\right) .  \notag
\end{eqnarray}

Let $x\in I_{N}.$ Since $x-t$
$\in I_{N},$ for $t\in I_{N}$ and
\begin{equation*}
v\left( n\right) +v^{\ast }\left( n\right) \leq c\left( \left\vert
n\right\vert -\left\langle n\right\rangle \right) =c\rho \left( n\right)
\end{equation*}
if we apply (\ref{var}) we get that
\begin{eqnarray}
\left\vert \frac{M_{\left\langle n\right\rangle }^{1/p-1}S_{n}a\left(
x\right) }{M_{\left\vert n\right\vert }^{1/p-1}}\right\vert 
&\leq &\frac{M_{\left\langle n\right\rangle }^{1/p-1}M_{N}^{1/p}}{%
M_{\left\vert n\right\vert }^{1/p-1}}\int_{I_{N}}\left\vert D_{n}\left(
t\right) \right\vert d\mu \left( t\right) \\ \notag
&\leq &\frac{%
M_{\left\langle n\right\rangle }^{1/p-1}M_{N}^{1/p}\left( v\left( n\right)
+v^{\ast }\left( n\right) \right) }{M_{\left\vert n\right\vert }^{1/p-1}} \\
&\leq &\frac{cM_{\left\langle n\right\rangle }^{1/p-1}M_{N}^{1/p}\left(
\left\vert n\right\vert -\left\langle n\right\rangle \right) }{M_{\left\vert
n\right\vert }^{1/p-1}}  \notag \\
&\leq & \frac{cM_{N}^{1/p}\rho \left( n\right) }{2^{\rho\left( n\right)
\left( 1/p-1\right) }}  \notag
\end{eqnarray}
and
\begin{equation} \label{11b}
\int_{I_{N}}\left\vert \frac{M_{\left\langle n\right\rangle
}^{1/p-1}S_{n}a\left( x\right) }{M_{\left\vert n\right\vert }^{1/p-1}}%
\right\vert ^{p}d\mu \left( x\right) \leq \frac{\rho ^{p}\left( n\right) }{%
2^{\rho \left( n\right) \left( 1-p\right) }}<c_{p}<\infty .  
\end{equation}

Let $x\in I_{s}\setminus I_{s+1},\,0\leq s\leq N-1<\left\langle
n\right\rangle $ or $0\leq s\leq\left\langle n\right\rangle \leq N-1.$ Then $x-t$
$\in I_{s}\setminus I_{s+1},$ for $t\in I_{N}$. By combining (\ref{3dn})
and (\ref{2dn}) we get that
\begin{equation*}
D_{n}\left( x-t\right) =0
\end{equation*}
and
\begin{equation} \label{11bbaa}
\left\vert \frac{M_{\left\langle n\right\rangle }^{1/p-1}S_{n}a}{%
M_{\left\vert n\right\vert }^{1/p-1}}\right\vert =0.
\end{equation}

Let $x\in I_{s}\setminus I_{s+1},\,0\leq \left\langle
n\right\rangle < s\leq N-1$ or $0\leq \left\langle n\right\rangle <s \leq N-1.$ Then $x-t\in I_{s}\setminus
I_{s+1},$ for $t\in I_{N}$. By applying (\ref{estdn}) we get that%
\begin{eqnarray}  \label{12}
&&\left\vert \frac{M_{\left\langle n\right\rangle }^{1/p-1}S_{n}a\left(
x\right) }{M_{\left\vert n\right\vert }^{1/p-1}}\right\vert \\
&\leq& \frac{c_p M_{\left\langle n\right\rangle }^{1/p-1}M_{N}^{1/p}}{%
M_{\left\vert n\right\vert }^{1/p-1}}\frac{M_{s}}{M_{N}}  \notag \\
&=& c_p M_{\left\langle n\right\rangle }^{1/p-1}M_{s}.  \notag
\end{eqnarray}

By combining (\ref{1.1}), (\ref{11bbaa}) and (\ref{12}) we have that
\begin{eqnarray}
&&\int_{\overline{I_{N}}}\left\vert \frac{M_{\left\langle n\right\rangle
}^{1/p-1}S_{n}a}{M_{\left\vert n\right\vert }^{1/p-1}}\right\vert ^{p}d\mu \\
&=& \overset{N-1}{\underset{s=0}{\sum }}%
\int_{I_{s}\setminus I_{s+1}}\left\vert \frac{M_{\left\langle n\right\rangle
}^{1/p-1}S_{n}a}{M_{\left\vert n\right\vert }^{1/p-1}}\right\vert ^{p}d\mu  \notag \\
&\leq& c_p \overset{N-1}{\underset{s=\left\langle n\right\rangle }{\sum }}%
\int_{I_{s}\setminus I_{s+1}}\left\vert M_{\left\langle n\right\rangle
}^{1/p-1}M_{s}\right\vert ^{p}d\mu  \notag \\
&=& c_p \overset{N-1}{\underset{s=\left\langle n\right\rangle }{\sum }} \frac{c_p {M_{\left\langle n\right\rangle }^{1-p}}}{{M_{s}^{1-p}}}
\leq c_{p}<\infty .  \notag
\end{eqnarray}

Hence, the proof of part a) is complete.

Under condition (\ref{12e}), there exists a sequence $\left\{ \alpha _{k}:%
\text{ }k\in \mathbb{N}\right\} \subset \left\{ n_{k}:\text{ }k\in \mathbb{N}%
\right\} ,$ such that
\begin{equation}
\sum_{\eta =0}^{\infty }\frac{M_{\left\langle n_{\eta }\right\rangle
}^{\left( 1-p\right) /2}\Phi _{n_{\eta }}^{p/2}}{M_{\left\vert \alpha _{\eta
}\right\vert }^{\left( 1-p\right) /2}}<\infty .  \label{12f}
\end{equation}

We note that such increasing sequence $\left\{ \alpha _{k}:k\in \mathbb{N}%
\right\} $ which satisfies condition (\ref{12f}) can be constructed.

Let $f=\left( f_{n},n\in \mathbb{N}\right) $ be a martigale from the Lemma %
\ref{example2.6}, where
\begin{equation}  \label{fna}
\lambda _{k}=\frac{M_{\left\langle \alpha _{k}\right\rangle }^{\left(
1/p-1\right) /2}\Phi _{\alpha _{k}}^{1/2}}{M_{\left\vert \alpha
_{k}\right\vert }^{\left( 1/p-1\right) /2}}.
\end{equation}

In view of (\ref{fna}) we conclude that (\ref{2aa}) is satisfied and by
using Lemma \ref{example2.6} we obtain that\ $f\in H_{p}.$

By now using (\ref{11AA}) with $\lambda _{k}$ defined by (\ref{fna}) we get
that
\begin{eqnarray}  \label{6aaa}
&&\frac{S_{\alpha _{k}}f}{\Phi _{\alpha _{k}}} \\
&=&\frac{1}{\Phi _{\alpha _{k}}}\sum_{\eta =0}^{k-1}M_{\left\vert \alpha
_{\eta }\right\vert }^{\left( 1/p-1\right) /2}M_{\left\langle \alpha _{\eta
}\right\rangle }^{\left( 1/p-1\right) /2}\Phi _{\alpha _{\eta }}^{1/2}\left(
M_{\left\vert \alpha _{\eta }\right\vert +1}-D_{M_{\left\vert \alpha _{\eta
}\right\vert }}\right)  \notag \\
&+&\frac{M_{\left\vert \alpha _{k}\right\vert }^{\left( 1/p-1\right)
/2}M_{\left\langle \alpha _{k}\right\rangle }^{\left( 1/p-1\right)
/2}D_{\alpha _{k}-M_{\left\vert \alpha _{k}\right\vert }}}{\Phi _{\alpha
_{k}}^{1/2}}=I+II.  \notag
\end{eqnarray}

According to (\ref{12f}) we can write that
\begin{eqnarray}  \label{8a}
&&\left\Vert I\right\Vert _{L_{p,\infty }}^{p} \\
&&\leq \frac{1}{\Phi _{\alpha _{k}}^{p}}\sum_{\eta =0}^{\infty }\frac{%
M_{\left\langle \alpha _{\eta }\right\rangle }^{\left( 1-p\right) /2}\Phi
_{\alpha _{\eta }}^{p/2}}{M_{\left\vert \alpha _{\eta }\right\vert }^{\left(
1-p\right) /2}}\left\Vert M_{\left\vert \alpha _{\eta }\right\vert }^{\left(
1/p-1\right) }\left( M_{\left\vert \alpha _{\eta }\right\vert
+1}-D_{M_{\left\vert \alpha _{\eta}\right\vert }}\right) \right\Vert
_{L_{p,\infty }}^{p}  \notag \\
&\leq & \frac{1}{\Phi _{\alpha _{k}}^{p}}\sum_{\eta =0}^{\infty }\frac{%
M_{\left\langle \alpha _{\eta }\right\rangle }^{\left( 1-p\right) /2}\Phi
_{\alpha _{\eta }}^{p/2}}{M_{\left\vert \alpha _{\eta }\right\vert
}^{\left(1-p\right) /2}}<\frac{c}{\Phi _{\alpha _{k}}^{p}}\leq c<\infty .
\notag
\end{eqnarray}

Let $x\in I_{\left\langle \alpha _{k}\right\rangle }\setminus
I_{\left\langle \alpha _{k}\right\rangle +1}.$ If we apply (\ref{estdn0}) we conclude that%
\begin{eqnarray}  \label{12aa}
\left\vert II\right\vert&=&\frac{M_{\left\vert \alpha _{k}\right\vert
}^{\left( 1/p-1\right) /2}M_{\left\langle \alpha _{k}\right\rangle }^{\left(
1/p-1\right) /2}D_{\alpha _{k}-M_{\left\vert \alpha _{k}\right\vert }}}{\Phi
_{\alpha _{k}}^{1/2}} \\
&\geq &\frac{M_{\left\vert \alpha _{k}\right\vert }^{\left( 1/p-1\right)
/2}M_{\left\langle \alpha _{k}\right\rangle }^{\left( 1/p+1\right) /2}}{\Phi
_{\alpha _{k}}^{1/2}}.  \notag
\end{eqnarray}

By combining (\ref{8a}) and (\ref{12aa}) for the sufficiently large $k $, we can write that
\begin{eqnarray*}
&&\left\Vert \frac{S_{\alpha _{k}}f}{\Phi _{\alpha _{k}}}\right\Vert
_{L_{p,\infty }}^{p} \\
&\geq &\left\Vert II\right\Vert _{L_{p,\infty }}^{p}-\left\Vert I\right\Vert
_{L_{p,\infty }}^{p}  \notag \\
&\geq &\frac{1}{2}\left\Vert II\right\Vert _{L_{p,\infty }}^{p}  \notag \\
&\geq &\frac{cM_{\left\vert \alpha _{k}\right\vert }^{\left( 1-p\right)
/2}M_{\left\langle \alpha _{k}\right\rangle }^{\left( 1+p\right) /2}}{\Phi
_{\alpha _{k}}^{p/2}}\mu \left\{ x\in G_{m}:\text{ }\left\vert II\right\vert
\geq \frac{cM_{\left\vert \alpha _{k}\right\vert }^{\left( 1/p-1\right)
/2}M_{\left\langle \alpha _{k}\right\rangle }^{\left( 1/p+1\right) /2}}{%
\Phi_{\alpha _{k}}^{1/2}}\right\}  \notag \\
&\geq &\frac{cM_{\left\vert \alpha _{k}\right\vert }^{\left( 1-p\right)
/2}M_{\left\langle \alpha _{k}\right\rangle }^{\left( 1+p\right) /2}}{\Phi
_{\alpha _{k}}^{p/2}}\mu \left\{ I_{\left\langle \alpha _{k}\right\rangle
}\backslash I_{\left\langle \alpha _{k}\right\rangle +1}\right\}  \notag \\
&\geq & \frac{cM_{\left\vert \alpha _{k}\right\vert }^{\left( 1-p\right) /2}%
}{ M_{\left\langle \alpha _{k}\right\rangle }^{\left( 1-p\right) /2}\Phi
_{\alpha _{k}}^{p/2}}\rightarrow \infty ,\text{ as \ }k\rightarrow \infty .
\notag
\end{eqnarray*}

Thus, also part b) is proved so the proof is complete.
\end{proof}

Next, we present equivalent characterizations of boundedness  of the subsequences of partial sums with respect to the Vilenkin system of $f\in H_{p}$ martingales  in terms of measurable properties of a Dirichlet kernel:

\begin{corollary}
\label{ndn}a) Let $0<p<1$ and $f\in H_{p}$. Then there exists an \textit{%
absolute constant }$c_{p},$ depending only on $p,$ such that%
\begin{equation*}
\text{ }\left\Vert S_{n}f\right\Vert _{H_{p}}\leq c_{p}\left( n\mu \left\{
\text{supp}\left( D_{n}\right) \right\} \right) ^{1/p-1}\left\Vert
f\right\Vert _{H_{p}}.
\end{equation*}

b) Let $0<p<1\ \ $ and $\left\{ n_{k}:\text{ }k\in\mathbb{N}\right\} $ be an increasing sequence of  nonnegative integers such that
\begin{equation}
\sup_{k\in \mathbb{N}}n_{k}\mu \left\{ \text{supp}\left( D_{n_{k}}\right)
\right\} =\infty  \label{nkdnk}
\end{equation}
and $\left\{ \Phi _{n}:n\in\mathbb{N}\right\} $ be any nondecreasing
sequence, satisfying the condition
\begin{equation*}
\overline{\underset{k\rightarrow \infty }{\lim }}\frac{\left( n_{k}\mu
\left\{ \text{supp}\left( D_{n_{k}}\right) \right\} \right) ^{1/p-1}}{\Phi
_{n_{k}}}=\infty .
\end{equation*}%
Then there exists a martingale $f\in H_{p},$ such that
\begin{equation*}
\underset{k\in \mathbb{N}}{\sup }\left\Vert \frac{S_{n_{k}}f}{\Phi _{n_{k}}}%
\right\Vert _{L_{p,\infty }}=\infty .
\end{equation*}
\end{corollary}

\begin{remark}
\label{remarkndn}Corollary \ref{ndn} shows that when $0<p<1$ the main reason of divergence of partial sums of Vilenkin-Fourier series is unboudedness of Fourier coefficients, but in the case when measure of supp of $n_{k}$-th Dirichlet kernels tends to zero, then the divergence rate drops and in the case when it is maximally small
\begin{equation*}
\mu\left(suppD_{n_{k}}\right)=O\left( \frac{1}{M_{\left\vert n_{k}\right\vert }}\right) ,%
\text{ as \ }\text{ \ \ \ }k\rightarrow \infty,\text{ \ \ \ } \left(
M_{\left\vert n_{k}\right\vert }<n_{k}\leq M_{\left\vert
n_{k}\right\vert+1}\right)
\end{equation*}%
we have convergence.
\end{remark}

\begin{proof}
By combining (\ref{3dn}) and (\ref{2dn}) we can write that%
\begin{equation*}
I_{\left\langle n\right\rangle }\backslash I_{\left\langle n\right\rangle
+1}\subset \text{supp}D_{n}\subset I_{\left\langle n\right\rangle }
\end{equation*}
and
\begin{equation*}
\frac{1}{2M_{\left\langle n\right\rangle }}\leq \mu \left\{ \text{supp}%
D_{n}\right\} \leq \frac{1}{M_{\left\langle n\right\rangle }}
\end{equation*}

Since $M_{\left\vert n\right\vert }\leq n<M_{\left\vert n\right\vert +1}$ we
immediately get that%
\begin{equation*}
\frac{M_{\left\vert n\right\vert }}{2M_{\left\langle n\right\rangle }}\leq
n\mu \left\{ \text{supp}\left( D_{n}\right) \right\} \leq \frac{\lambda
M_{\left\vert n\right\vert }}{M_{\left\langle n\right\rangle }},
\end{equation*}%
where $\lambda =\sup_{n\in \mathbb{N}}m_{n}.$

It follows that
\begin{equation*}
\frac{M_{\left\vert n\right\vert }^{1/p-1}}{2M_{\left\langle n\right\rangle
}^{1/p-1}}\leq \left( n\mu \left\{ \text{supp}\left( D_{n}\right) \right\}
\right) ^{1/p-1}\leq \frac{\lambda ^{1/p-1}M_{\left\vert n\right\vert
}^{1/p-1}}{M_{\left\langle n\right\rangle }^{1/p-1}}.
\end{equation*}

The proof follows by using these estimates in Theorem \ref{theorem4.15}.
\end{proof}

A number of special cases of our result are of particular interest and give both well-known and new information. We just give the following examples of such Corollaries (see Corollaries \ref{corollary4.5.0}-\ref{corollary4.5.2}):

\begin{corollary}
\label{corollary4.5.0}Let $0<p<1$, $f\in H_{p}$ and $\left\{ n_{k}:\text{ }k\in\mathbb{N}\right\} $ be an increasing sequence of  nonnegative integers. Then
\begin{equation*}
\left\Vert S_{n_{k}}f\right\Vert _{H_{p}}\leq c_{p}\left\Vert f\right\Vert
_{H_{p}}
\end{equation*}%
if and only if condition (\ref{ak0}) is satisfied.
\end{corollary}

\begin{proof}
It is easy to show that
\begin{equation*}
2^{\rho \left( n_{k}\right) }\leq \frac{M_{\left\vert n_{k}\right\vert }}{%
M_{\left\langle n_{k}\right\rangle }}\leq \lambda ^{\rho \left( n_{k}\right)
},
\end{equation*}%
where $\lambda =\sup_{n\in \mathbb{N}}m_{n}.$ It follows that%
\begin{equation*}
\sup_{k\in \mathbb{N}}\frac{M_{\left\vert n_{k}\right\vert }^{1/p-1}}{%
M_{\left\langle n_{k}\right\rangle }^{1/p-1}}<c<\infty
\end{equation*}%
if and only if (\ref{ak0}) holds, so the proof follows by using Theorem \ref%
{theorem4.15}.
\end{proof}

\begin{corollary}
\label{corollary4.5.1}Let $n\in \mathbb{N}$ and $0<p<1.$ Then there exists a
martingale $f\in H_{p},$ such that
\begin{equation}
\underset{n\in \mathbb{N}}{\sup }\left\Vert S_{M_{n}+1}f\right\Vert
_{L_{p,\infty }}=\infty .  \label{sn2n1}
\end{equation}
\end{corollary}

\begin{proof}
To prove Corollary \ref{corollary4.5.1} we only have to calculate that%
\begin{equation}  \label{1.11}
\left\vert M_{n}+1\right\vert =n,\text{ }\left\langle M_{n}+1\right\rangle =0
\end{equation}%
and
\begin{equation}  \label{1.111}
\rho \left( M_{n}+1\right) =n.\text{ }
\end{equation}%

By using Corollary \ref{corollary4.5.0} we obtain that there exists an
martingale $f\in H_{p}$ $\left( 0<p<1\right) ,$ such that (\ref{sn2n1})
holds.

The proof is complete.
\end{proof}

\begin{corollary}
\label{corollary4.5.2}Let $n\in \mathbb{N}$, $0<p\leq 1$ and $f\in H_{p}$.
Then
\begin{equation}
\left\Vert S_{M_{n}+M_{n-1}}f\right\Vert _{H_{p}}\leq c_{p}\left\Vert
f\right\Vert _{H_{p}}.  \label{sn2n2}
\end{equation}
\end{corollary}

\begin{proof}
Analogously to (\ref{1.11}) and (\ref{1.111}) we can write that%
\begin{equation*}
\left\vert M_{n}+M_{n-1}\right\vert =n,\text{ \ \ }\left\langle
M_{n}+M_{n-1}\right\rangle =n-1
\end{equation*}
and
\begin{equation*}
\rho \left( M_{n}+M_{n-1}\right)=1.
\end{equation*}

By using Corollary \ref{corollary4.5.0} we immediately get inequality (\ref%
{sn2n2}), for all $0<p\leq 1$.

The proof is complete.
\end{proof}

\begin{corollary}
\label{corollary4.5.222}Let $n\in \mathbb{N}$, $0<p\leq 1$ and $f\in H_{p}$.
Then
\begin{equation}
\left\Vert S_{M_{n}}f\right\Vert _{H_{p}}\leq c_{p}\left\Vert
f\right\Vert _{H_{p}}.  \label{sn2n2}
\end{equation}
\end{corollary}

\begin{proof}
Analogously to (\ref{1.11}) and (\ref{1.111}) we can write that%
\begin{equation*}
\left\vert M_{n}\right\vert =n,\text{ \ \ }\left\langle
M_{n}\right\rangle =n \text{ \ and \ } \rho \left( M_{n}\right)=0.
\end{equation*}

By using Corollary \ref{corollary4.5.0} we immediately get inequality (\ref%
{sn2n2}), for all $0<p\leq 1$.

The proof is complete.
\end{proof}

\section{necessary and sufficient condition for the
convergence of partial sums in terms of modulus of continuity}

The main result of this section reads:

\begin{theorem}
\label{theorem4.7a}Let $0<p<1,$ $f\in H_{p}\ \ $and $M_{k}<n\leq M_{k+1}.$
Then there exists an absolute constant $c_{p}$, depending only on $p,$ such
that
\begin{equation} \label{mag0a}
\left\Vert S_{n}f-f\right\Vert _{H_{p}}\leq \frac{c_{p}M_{\left\vert
n\right\vert }^{1/p-1}}{M_{\left\langle n\right\rangle }^{1/p-1}}\omega
\left( \frac{1}{M_{k}},f\right)_{H_{p}(G_m)} ,\ \ \ \left( 0<p<1\right) .
\end{equation}

Moreover, if $\{n_{k}:k\in%
\mathbb{N}\}$\textit{\ be increasing sequence of nonnegative integers such
that}
\begin{equation*}
\omega \left( \frac{1}{M_{\left\vert n_{k}\right\vert }},f\right)_{H_{p}(G_m)}
=o\left( \frac{M_{\left\langle n_{k}\right\rangle }^{1/p-1}}{M_{\left\vert
n_{k}\right\vert }^{1/p-1}}\right) ,\text{ as \ }k\rightarrow \infty,
\end{equation*}
then
\begin{equation}
\left\Vert S_{n_{k}}f-f\right\Vert _{H_{p}}\rightarrow 0,\,\,\,\text{as}%
\,\,\,k\rightarrow \infty.  \label{con1pa}
\end{equation}

b) Let $\{n_{k}:k\in\mathbb{N}\}$ be an increasing sequence of nonnegative
integers such that condition (\ref{ak}) is satisfied. Then
there exists a martingale $f\in H_{p}$ and a subsequence $\{\alpha _{k}:k\in%
\mathbb{N}\}\subset \{n_{k}:k\in\mathbb{N}\},$ for which
\begin{equation}  \label{mag}
\omega \left( \frac{1}{M_{\left\vert \alpha _{k}\right\vert }}%
,f\right)_{H_{p}(G_m)} =O\left( \frac{M_{\left\langle \alpha _{k}\right\rangle }^{1/p-1}}{%
M_{\left\vert \alpha _{k}\right\vert }^{1/p-1}}\right),\text{ as \ }%
k\rightarrow \infty
\end{equation}
\textit{and}
\begin{equation}
\limsup\limits_{k\rightarrow \infty }\left\Vert S_{\alpha _{k}}
f-f\right\Vert _{L_{p,\infty }}>c>0,\,\,\,\text{as\thinspace \thinspace
\thinspace }k\rightarrow \infty .  \label{cond9}
\end{equation}
\end{theorem}

\begin{proof}
Let $0<p<1.$ Then, by using Theorem \ref{theorem4.15}, we immediately get
that%
\begin{eqnarray*}
&&\left\Vert S_{n}f-f\right\Vert _{H_{p}}^{p} \\
& \leq & \left\Vert S_{n}f-S_{M_{k}}f\right\Vert _{H_{p}}^{p}+\left\Vert
S_{M_{k}}f-f\right\Vert _{H_{p}}^{p}  \notag \\
&=&\left\Vert S_{n}\left( S_{M_{k}}f-f\right) \right\Vert
_{H_{p}}^{p}+\left\Vert S_{M_{k}}f-f\right\Vert _{H_{p}}^{p}  \notag \\
&\leq & \left(\frac{c_{p}M_{\left\vert n\right\vert }^{1/p-1}}{%
M_{\left\langle n\right\rangle }^{1/p-1}}+1\right) \omega _{H_{p}}^{p}\left(
\frac{1}{M_{k}},f\right).  \notag
\end{eqnarray*}%
and%
\begin{equation} \label{4.18sss}
\left\Vert S_{n}f-f\right\Vert _{H_{p}}\leq \frac{c_{p}M_{\left\vert
n\right\vert }^{1/p-1}}{M_{\left\langle n\right\rangle }^{1/p-1}}\omega
\left( \frac{1}{M_{k}},f\right)_{H_{p}(G_m)} .
\end{equation}

According to condition (\ref{mag0a}) if we  apply estimate (\ref{4.18sss}) we immediately get that (\ref{con1pa})
holds.

For the proof of part b) we first note that under the conditions of part b)
of Theorem \ref{theorem4.7a}, there exists $\{{\alpha _{k}:k\in \mathbb{N}}%
\}\subset\{{n_{k}:k\in \mathbb{N}}\}$, such that $\ $
\begin{eqnarray}  \label{4.18}
\ \frac{M_{\left\vert \alpha _{k}\right\vert }}{M_{\left\langle \alpha
_{k}\right\rangle }}\uparrow \infty ,\,\,\,\,\text{as}\ \ k\rightarrow \infty
\\
\frac{M_{\left\vert \alpha _{k}\right\vert }^{2\left( 1/p-1\right) }}{%
M_{\left\langle \alpha _{k}\right\rangle }^{2\left( 1/p-1\right) }}\leq
\frac{M_{\left\vert \alpha _{k+1}\right\vert }^{1/p-1}}{M_{\left\langle
\alpha _{k+1}\right\rangle }^{1/p-1}}.  \label{4.18s}
\end{eqnarray}

Let $f=\left( f_{n},n\in \mathbb{N}\right) $ be a martigale from the Lemma %
\ref{example2.6}, where
\begin{equation}  \label{4.18ss}
\lambda_{k}=\frac{\lambda M_{\left\langle \alpha _{k}\right\rangle }^{1/p-1}}{%
M_{\left\vert \alpha _{k}\right\vert }^{1/p-1}}.
\end{equation}

If we apply (\ref{4.18}) and (\ref{4.18s}) in the case when $ \lambda_{k} $ are defined by (\ref{4.18ss}) we conclude that (\ref{2aa}) is satisfied and by
using Lemma \ref{example2.6} we obtain that\ $f\in H_{p}.$

By using (\ref{4.18ss}) with $\lambda _{k}$ defined by (\ref{4.18ss}) we get
that%
\begin{eqnarray}  \label{4.21}
&&\omega (\frac{1}{M_{\left\vert \alpha _{k}\right\vert }},f)_{H_{p}(G_m)} \leq
\sum\limits_{i=k}^{\infty }\frac{M_{\left\langle \alpha
_{i}\right\rangle}^{1/p-1}}{M_{\left\vert \alpha _{i}\right\vert }^{1/p-1}}= O\left( \frac{M_{\left\langle \alpha _{k}\right\rangle }^{1/p-1}}{%
M_{\left\vert \alpha _{k}\right\vert }^{1/p-1}}\right), \text{ as \ }%
k\rightarrow \infty .
\end{eqnarray}

By applying (\ref{11AA}) with $\lambda _{k}$ defined by (\ref{4.18ss}) we
get that
\begin{eqnarray*}
&& S_{\alpha _{k}}f = S_{M_{\left\vert \alpha _{k}\right\vert }}+{M_{\left\langle \alpha _{k}\right\rangle  }^{1/p-1}\psi _{M_{\left\vert \alpha _{k}\right\vert }}D_{j-M_{_{\left\vert \alpha _{k}\right\vert }}}} 
\end{eqnarray*}

In view of (\ref{estdn0}) we can conclude that%
\begin{equation*}
\left\vert D_{\alpha _{k}-M_{\left\vert \alpha _{k}\right\vert}}\right\vert \geq M_{\left\langle \alpha
_{k}\right\rangle },\text{ \ for \ } I_{\left\langle \alpha
_{k}\right\rangle}\setminus I_{\left\langle \alpha _{k}\right\rangle +1}
\end{equation*}
and
\begin{eqnarray} \label{4.2222}
&& M_{\left\langle \alpha
_{k}\right\rangle }\mu \left\{ x\in G_{m}:\text{ }\left\vert D_{\alpha _{k}-M_{\left\vert \alpha _{k}\right\vert}}\right\vert
\geq M_{\left\langle \alpha
_{k}\right\rangle }\right\}   \\ \notag
&&M_{\left\langle \alpha
_{k}\right\rangle }\mu \left\{ I_{\left\langle \alpha _{k}\right\rangle
}\backslash I_{\left\langle \alpha _{k}\right\rangle +1}\right\}\geq M_{\left\langle \alpha
_{k}\right\rangle }^{1-p}.
\end{eqnarray}
If we combine Corollary  \ref{corollary4.5.222} and (\ref{4.2222}), for the sufficiently large $ k $, we can write that
\begin{eqnarray}
&&\Vert S_{\alpha _{k}}f-f\Vert _{L_{p,\infty }} 
\geq  M_{\left\langle \alpha _{k}\right\rangle }^{1/p-1}\Vert D_{\alpha
_{k}}\Vert _{L_{p,\infty }}-\Vert S_{M_{\left\vert \alpha _{k}\right\vert }}f-f\Vert_{L_{p,\infty }}  \notag \\
&\geq & \frac{M_{\left\langle \alpha _{k}\right\rangle }^{1/p-1}\Vert D_{\alpha_{k}}\Vert _{L_{p,\infty }}}{2}\geq c,\text{ \ \ as \ \ } k\rightarrow \infty .  \notag
\end{eqnarray}

The proof is complete.
\end{proof}

Next, we present simple consequence of Theorem \ref{theorem4.7a} which was proved in Tephnadze \protect\cite{tep5}:
\begin{corollary}
\label{corollary4.7}a) Let $0<p<1,$ $f\in H_{p}$ and%
\begin{equation*}
\omega \left( \frac{1}{M_{n}},f\right)_{H_{p}(G_m)} =o\left( \frac{1}{%
M_{n}^{1/p-1}}\right) ,\text{ as }n\rightarrow \infty .
\end{equation*}%
Then
\begin{equation*}
\left\Vert S_{k}f-f\right\Vert _{H_{p}}\rightarrow 0,\,\,\,\text{as \ }%
\,\,\,k\rightarrow \infty .
\end{equation*}%
\textit{b) For every }$0<p<1$\textit{\ there exists martingale }$f\in H_{p}$%
\textit{, for which}
\begin{equation*}
\omega \left( \frac{1}{M_{n}},f\right) =O\left( \frac{1}{%
M_{n}^{1/p-1}}\right)_{H_{p}(G_m)} ,\text{ \ as \ \ \ }n\rightarrow \infty
\end{equation*}%
\textit{and}
\begin{equation*}
\left\Vert S_{k}f-f\right\Vert _{L_{p,\infty }}\nrightarrow 0,\,\,\,\text{%
as\thinspace \thinspace \thinspace }k\rightarrow \infty .
\end{equation*}
\end{corollary}

Finally, we present equivalent conditions for the modulus of continuity in terms of measurable properties of a Dirichlet kernel, which provide boundedness  of the subsequences of partial sums with respect to the Vilenkin system of $f\in H_{p}$ martingales:

\begin{corollary}
\label{Corollary4.7a}Let $0<p<1,$ $f\in H_{p}\ \ $and $M_{k}<n\leq M_{k+1}.$
Then there exists an absolute constant $c_{p}$, depending only on $p,$ such
that
\begin{equation*}
\left\Vert S_{n}f-f\right\Vert _{H_{p}}\leq c_{p}\left( n\mu \left( \text{%
supp}D_{n}\right) \right) ^{1/p-1}\omega _{H_{p}}\left( \frac{1}{M_{k}}%
,f\right) ,\ \ \ \left( 0<p<1\right)
\end{equation*}

Moreover, if $\{n_{k}:k\in\mathbb{N}\}$\textit{\ be a sequence of nonnegative integers such that}
\begin{equation*}
\omega\left( \frac{1}{M_{\left\vert n_{k}\right\vert }},f\right)_{H_{p}(G_m)}
=o\left( \frac{1}{\left( n_{k}\mu \left( \text{supp}D_{n_{k}}\right) \right)
^{1/p-1}}\right) ,\text{ as \ }k\rightarrow \infty,
\end{equation*}%
then (\ref{con1pa}) holds.

b) Let $\{n_{k}:k\in\mathbb{N}\}$ be an increasing sequence of nonnegative
integers such that condition (\ref{ak}) is satisfied. Then there exists a
martingale $f\in H_{p} $ and a subsequence $\{\alpha _{k}:k\in\mathbb{N}%
\}\subset \{n_{k}:k\in\mathbb{N}\},$ for which
\begin{equation*}
\omega\left( \frac{1}{M_{\left\vert \alpha _{k}\right\vert }}%
,f\right)_{H_{p}(G_m)} =O\left( \frac{1}{\left( \alpha _{k}\mu \left( \text{supp}%
D_{\alpha _{k}}\right) \right) ^{1/p-1}}\right) \text{\ },\,\,\,\text{as\thinspace \thinspace
\thinspace }k\rightarrow \infty
\end{equation*}%
\textit{and}
\begin{equation*}
\limsup\limits_{k\rightarrow \infty }\left\Vert S_{\alpha
_{k}}f-f\right\Vert _{L_{p,\infty }}>c>0,\,\,\,\text{as\thinspace \thinspace
\thinspace }k\rightarrow \infty .
\end{equation*}
\end{corollary}

\end{document}